\documentclass[12pt,a4paper]{amsart}
\usepackage[T1]{fontenc}
\usepackage{amsmath,amsthm,amsfonts,amssymb,amscd}
\usepackage[colorlinks,
            linkcolor=blue,       
            anchorcolor=blue,  
            citecolor=blue,        
            ]{hyperref}
\usepackage{graphicx}
\usepackage{quiver}                     
\usepackage{subcaption}                 
\newtheorem{theorem}{Theorem}[section]  
\newtheorem{pro}[theorem]{Proposition}  
\newtheorem{cor}[theorem]{Corollary}    
\newtheorem{lem}[theorem]{Lemma}        
\newtheorem{remark}[theorem]{Remark} 

\title{Infinite class field tower with small root discriminant}
\author{Qi Liu}
\author{Zugan Xing*}
\date{2024.May}

\begin{document}


\address{(Qi Liu) Department of Mathematics, Capital Normal University, Beijing,
China}
\email{QiLiu67@aliyun.com}
\address{(Zugan Xing*) Department of Mathematics, Capital Normal University, Beijing,
China}
\email{xingzugan@aliyun.com}
\setlength{\belowcaptionskip}{-5pt}
\setlength{\abovecaptionskip}{-5pt}

\begin{abstract}
We generalize Schoof's theorem in 1986 and apply this to construct a class of Kummer extensions of the cyclotomic fields with infinite class tower. As an application, we give some number fields with a small root discriminant, which has an infinite $p$-class field tower when $p=3, 5, 7$.
\end{abstract}

\maketitle


\section{Introduction}
Let $K$ be a number field. We define a sequence of extensions of $K$ as follows: Let $K_{i+1}$ be the maximal abelian unramified extension of $K_i$, for $i\geq 0$, and $K_0=K$. Then, we call this sequence of fields the Hilbert class field tower of $K$. The Hilbert class field tower of $K$ is called finite if there exists $i\geq 0$ such that $K_{i+1}=K_i$. Otherwise, it is named infinite. For a rational prime $p$, the $p$-Hilbert class field of $K$ is defined as the maximal abelian unramified $p$-extension of $K$. Likewise, the $p$-class field tower of $K$ can be defined accordingly.

Historically, Golod and Shafarevich \cite{golod1964class} demonstrated in 1964 that certain quadratic number fields,  $\mathbb{Q}(\sqrt{-3\cdot5\cdot7\cdot11\cdot13\cdot17\cdot19})$, possess an infinite 2-class field tower. Schoof later complemented this seminal finding in 1986, provided a robust method for constructing number fields with infinite class field towers, and produced infinitely many quadratic number fields admitting an infinite 2-class field tower.

The existence of an infinite class field tower is closely tied to the size of the root discriminant. Odlyzko's bounds \cite{odlyzko1990bounds} imply that any number field with infinite class field tower must have root discriminant at least 22.3. Furthermore, assuming the GRH, the bound becomes 44.3. It is an interesting task to find a number field with a small root discriminant having an infinite $p$-class field tower.

Martinet \cite{martinet1978tours} found a number field $K=\mathbb Q(\zeta_{11}+\zeta_{11}^{-1},\sqrt{-46})$ with root discriminant $\approx 92.4$, having infinite class field tower. After that, If we fix a finite set $S$ of finite places of K, Hajir, and Maire \cite{hajir2002tamely} gave an example with root discriminant $\approx82.1$, having infinite S-tamely ramified 2-tower. Recently, Hajir, Maire, and Ramakrishna \cite{hajir2021cutting} used the "cutting" method to improve the root discriminant record of the former to 78.427. 

For $p\geq 3$, Leshin \cite{leshin2014infinite} applies School's method to prove that a class of $\mathbb Z/p\mathbb Z \rtimes \mathbb Z/(p-1)\mathbb Z$  extensions of $\mathbb Q$ with infinite class field tower, and then finds that $\mathbb Q(\zeta_3,\sqrt[3]{79\cdot 97})$ has infinite 3-class field tower, with root discriminant $\approx 1400.4$. We continue the work of Leshin, our main result is the following.

\begin{theorem}\label{thm 1.1}
Suppose $m, n$ are positive integers such that $m$ divides $n$. Let $h$ denote the class number of $\mathbb{Q}(\zeta_n, \sqrt[m]{p})$. If $h$ is greater than or equal to $\frac{12}{\varphi(n)m} + \frac{8}{\varphi(n)}$, where \(\varphi(n)\) is the Euler's totient function of $n$ and $p$ is a prime, then there exist infinitely many prime numbers $q$ such that $\mathbb{Q}(\zeta_n, \sqrt[m]{pq})$ has an infinite class field tower.
\end{theorem}

As an application, we find a new Kummer extension of a cyclotomic number field with an infinite 3-class field tower and the small root discriminant.
\begin{cor}
    The number field $\mathbb Q(\zeta_{9},\sqrt[3]{7\cdot 181})$ has infinite 3-class field tower with root discriminant $\approx 776.7$.
\end{cor}
Moreover, we also obtain some new number fields, with infinite p-class field and small root discriminant, where $p=5, 7$.
\begin{cor}
The number field $\mathbb Q(\zeta_{40},\sqrt[5]{3\cdot 41})$ has infinite 5-class field tower with root discriminant $\approx 1196.2$.
\end{cor}
\begin{cor}
The number field $\mathbb Q(\zeta_{35},\sqrt[7]{5\cdot 71})$ has infinite 7-class field tower with root discriminant $\approx 1608.8$.
\end{cor}

\section{The extension of Shcoof's Criterion}

Let $K$ be a number field, with $\mathcal{O}_{K}$ denoting its ring of integers, $E_{K}$ the units of this ring, $C l_{K}$ the class group of $\mathcal{O}_{K}$, and $h_{K}$ the order of $C l_{K}$. Suppose \(K/k\) is a finite Galois extension with Galois group \(\Delta =\text{Gal}(K/k)\), and \(\mathfrak{p}\) is a prime of \(k\). Given a \(\mathbb{Z}\)-module \(M\) and rational prime number \(p\), we denote by \(d_p(M)\) the dimension over \(\mathbb{F}_p\) of \(\mathbb{F}_p \otimes_{\mathbb Z} M\). In 1986, Schoof proved the following Classic conclusion.

\begin{theorem} [Schoof, \cite{schoof1986infinite}]
Suppose \(\Delta\) is a cyclic group of prime order \(p\), and let $\rho$ denote the number of primes (both finite and infinite) of $K$ that ramify in $k$. Then $K$ has infinite $p$-class field tower if
$$
\rho \geq 3+d_p\left(E_k /\left(E_k \cap N_{U_K / U_k} U_K\right)\right)+2 \sqrt{d_p\left(E_K\right)+1},
$$  
where  $U_{K}$ is the idèle units of  \(K\).
\end{theorem}
    
In this section, we generalize Schoof's criterion. We first review two basic facts of \(d_p Cl_K\). 

\begin{pro}[\cite{schoof1986infinite}, Proposition 3.1] \label{Prop:1} 
The p-rank of  $C l_{K}$  satisfies
\begin{center}
$    d_{p} C l_{K} \geq d_{p} \hat{H}^{0}\left(\Delta, U_{K}\right)-d_{p} \Delta /[\Delta, \Delta]-d_{p} E_{k} /\left(E_{k} \cap N_{U_K /U_k} U_{K}\right) .$
\end{center}
\end{pro}

\begin{theorem}[\cite{cassels1967algebraic}, p.235]\label{Thm:2}  
K  has an infinite  \(p\)-class field tower if 
\[d_{p} C l_{K} \geq 2+2 \sqrt{d_{p} E_{K}+1}.\] 
\end{theorem}

Now, let $S$ be the set of all (finite and infinite) primes of $k$, $R$ be the set of all finite primes that are ramified in $K$, and let $S_{\infty}$ denote the set of infinite primes of $k$. Fixed a prime $\mathfrak{P}$ of $K$ lying over $\mathfrak{p}\in S$, we denote by $I_{\mathfrak{P}}(K/k)$ and $D_{\mathfrak{P}}(K/k)$ the inertia group and decomposition group of $\mathfrak{P}$ in $\operatorname{Gal}(K/k)$.

\begin{pro}\label{prop 2.4}
Let \(p\) be a rational number such that \(p \mid |\Delta|\). We set $\delta_{p}=1$ if $p = 2$ and 0 otherwise. Let $G'$ be the commutator subgroup of $G$. Then $K$  has an infinite $p$-class field tower if
\[
\sum_{\substack{\mathfrak{p} \in R}} d_{p}I_{\mathfrak{P}}(K / k)/{D^{\prime}_{\mathfrak{P}}(K / k)}\geq d_{p}\Delta/[\Delta,\Delta]+2-r_1^{ram}(k)\delta_{p}+L_{p}
\]
where \[ L_p = d_p\left(E_k /\left(E_k \cap N_{U_K / U_k} U_K\right)\right)+2 \sqrt{d_p\left(E_K\right)+1}.\]
    \begin{proof}

By proposition \ref{Prop:1} and theorem \ref{Thm:2}, it suffices to show that: \begin{equation*}
    d_{p} \hat{H}^{0}\left(\Delta, U_{K}\right)\geq\sum_{\substack{\mathfrak{p} \in R}}d_{p}I_{\mathfrak{P}}(K / k)/{D^{'}_{\mathfrak{P}}(K / k)}+r_1^{ram}(k)d_{{p}} (\mathbb{Z} / 2 \mathbb{Z})
\end{equation*}  From the same argument as in the proof of (\cite{neukirch2013cohomology},8.1.2), we have 
\begin{equation*}
    \hat{H}^{0}\left(\Delta, U_{K}\right)=\prod_{\mathfrak{p} \in S_{\infty}} \hat{H}^{0}\left(G\left(K_{\mathfrak{P}} /k_{\mathfrak{p}}\right), K_{\mathfrak{P}}^{\times}\right) \prod_{\mathfrak{p} \in S\backslash S_{\infty}} \hat{H}^{0}\left(G\left(K_{\mathfrak{P}} / k_{\mathfrak{p}}\right), U_{\mathfrak{P}}\right).
\end{equation*}
If  $K_{\mathfrak{P}} / k_{\mathfrak{p}}$ is unramified, the unit group  $U_{\mathfrak{P}}$ is cohomologically trivial $G\left(K_{\mathfrak{P}} / k_{\mathfrak{p}}\right)$-modules. Hence,
\begin{equation*}
    \hat{H}^{0}(\Delta, U_{K})=\prod_{\mathfrak{p} \in S_{\infty}} \hat{H}^{0}\left(G\left(K_{\mathfrak{P}} /k_{\mathfrak{p}}\right), K_{\mathfrak{P}}^{\times}\right) \prod_{\substack{\mathfrak{p}\in R}}\hat{H}^{0}(G\left(K_{\mathfrak{P}} / k_{\mathfrak{p}}\right), U_{\mathfrak{P}})
\end{equation*}
By (\cite{gras2003class}, Theorem 1.4),
\[\hat{H}^{0}\left(G\left(K_{\mathfrak{P}} / k_{\mathfrak{p}}\right), U_{\mathfrak{P}}\right)=I\left(K^{Abel}_{\mathfrak{P}} / k_{\mathfrak{p}}\right)=I_{\mathfrak{P}}(K / k)/{D^{\prime}_{\mathfrak{P}}(K / k)},
\]
and 
\begin{equation*}
    d_{p}\hat{H}^{0}\left(G\left(K_{\mathfrak{P}} / k_{\mathfrak{p}}\right), U_{\mathfrak{P}}\right)= d_{p}I_{\mathfrak{P}}(K / k)/{D^{\prime}_{\mathfrak{P}}(K / k)}.
\end{equation*}
Therefore, 
\begin{align*}
     d_{p}\hat{H}^{0}\left(\Delta, U_{K}\right)&=\sum_{\mathfrak{p} \in S_{\infty}} d_{p}\hat{H}^{0}\left(G\left(K_{\mathfrak{P}} /k_{\mathfrak{p}}\right), K_{\mathfrak{P}}^{\times}\right) \\&+\sum_{\substack{p\in R}}d_{p}\hat{H}^{0}\left(G\left(K_{\mathfrak{P}} / k_{\mathfrak{p}}\right), U_{\mathfrak{P}}\right)\\
&\geq \sum_{\substack{\mathfrak{p}\in R}}d_{p}I_{\mathfrak{P}}(K / k)/{D^{\prime}_{\mathfrak{P}}(K / k)} +r_1^{ram} \cdot (k)d_{p}\mathbb{Z}/2\mathbb{Z}
\end{align*}
\end{proof}

\end{pro}

\section{The proof of the Main Theorem}

\subsection{Construct the unramified extension}
In this subsection, we set \(m\geq 2\) be an integer. 
Our construction is analogous to that of Schoof \cite{schoof1986infinite}, Theorem 3.4. The following proposition will be used to construct the unramified extension of certain fields for proving Theorem \ref{thm 1.1}.

\begin{theorem} [\cite{schinzel1977abelian}, Theorem 2] \label{thm 3.1}
Let $k$ be a field, and $m$ be a natural such that char $(k) \nmid m$. Denote with $\omega_m$ the number of the $m$-roots of unity in $k$. Then $\operatorname{Gal}\left(k\left(\zeta_m, \sqrt[m]{a}\right) / k\right)$ is abelian if and only if
$$
a^{\omega_m}=r^m \text { for some } r \in k \text {. }
$$  
\end{theorem}

\begin{lem}\label{lem 3.2}

(I). Let $a \in \mathbb{Q}_p$ with $p \nmid a$, \(p\) is an odd prime number, and n is a positive integer. Then
$$
a \in \mathbb{Q}_{p}(\zeta_{p^n})^{p^n} \Leftrightarrow a \in \mathbb{Q}_p^{p^n} \Leftrightarrow p^{n+1} \mid\left(a^{p-1}-1\right) .
$$

(II). If $n=2^{k}$, $k\in \mathbb{N}^{*}$, and $c$ is an odd number, then 
$$ c\in \mathbb{Q}_{2}(\zeta_{n})^{n} \text{ if } c \text{ is } n \text{-power in } \mathbb{Z} /2^{2k+1} \mathbb{Z}.$$
\end{lem}

\begin{proof}
    (I). The sufficiency is clear for the first equivalence conditions. Assume $a\in \mathbb{Q}_{p}(\zeta_{p^n})^{p^n}$, we have $\mathbb{Q}_{p}(\zeta_{p^n},\sqrt[p^n]{a})=\mathbb{Q}_{p}(\zeta_{p^n})$. Hence, the group $\operatorname{Gal}(\mathbb{Q}_{p}(\zeta_{p^n},\sqrt[p^n]{q}) /\mathbb{Q}_{p})$ is an abelian group. By Theorem \ref{thm 3.1}, we have 
$$ a^{\omega_{n}} = r^{n} \text{ for some } r\in \mathbb{Q}_{p}.$$
Since the field $\mathbb{Q}_{p}$ lacks $p^n$-th roots of unity beyond the identity, it follows that $\omega_{p^n}=1$. This establishes the necessity. The second equivalence is detailed in Theorem 5.4 of \cite{viviani2004ramification}.

    (II). Let $\alpha\in \mathbb{Z}$ be an odd number such that
$ \alpha^{n} \equiv c \bmod 2^{2k+1}$.
Let $\beta =\alpha+2^{2k+1}$ and $f(x) =x^{n}-c$. It is easy to check 
\( |f(\beta)|_{2}< |f'(\beta)|_{2}^{2}.\)
By Hensel's lemma, there exists a root of $f(x)$ in $\mathbb{Q}_{2}$.
\end{proof}

\begin{lem}\label{lem 3.3}
Let \(m,n, a\) be three integers and \(p\) be a prime number. Then \\
(I). The prime ideals above \(p\) of \(\mathbb{Q}(\zeta_n)\) is unramified in \(\mathbb{Q}(\zeta_n,\sqrt[m]{a})\) when \( p\nmid ma\);\\
(II). The prime ideals above \(p\) of \(\mathbb{Q}(\zeta_n)\) are totally ramified in \(\mathbb{Q}(\zeta_n,\sqrt[n]{p})\) when \( p \nmid n\).
\end{lem}

\begin{proof}
$(I).$ Observe that the relative discriminant of \(\mathbb{Q}(\zeta_n,\sqrt[m]{a})/\mathbb{Q}(\zeta_n)\) divides the norm of \(x^m-a\), which equals \(m^ma^{m-1}\). Therefore, if \((p,ma)=1\), then \(p\) is not ramified.

$(II).$ This can be inferred from Theorem 4.3 in \cite{viviani2004ramification}.
\end{proof}

Suppose $n = 2^{n_0}p_1^{n_1}\cdots p_l^{n_l} \in \mathbb{N}^{+}$, where $p_i$ for $1 \le i \le l$ are distinct odd primes, $n_i \in \mathbb{N}$ for $0 \le i \le l$, and $l \in \mathbb{N}^{+}$. Let $m$ be an integer such that $m \mid n$. Define $F = \mathbb{Q}(\zeta_n)$, where $\zeta_n$ is the $n$-th primitive root of unity, and let $s$ be a rational prime not dividing $n$. We denote by $H$ the Hilbert class field of $F(\sqrt[m]{s})$.

\begin{pro}\label{prop 3.4}
    Let $s$ be an odd prime number. If $t$ is a rational prime that splits completely in $H\left(\zeta_{n \cdot P(n)\cdot 2^{2n_0}}\right)$ where \(P(n)\) is the product of distinct prime factors of \(n\), then the field $F(\sqrt[m]{st})$, denoted by $K$, satisfies $F(\sqrt[m]{s}, \sqrt[m]{t}) / K$ being unramified everywhere.
\end{pro}
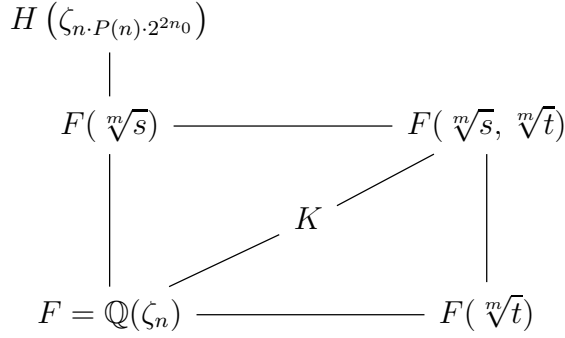
\begin{figure}[htbp]
\begin{tikzcd}[sep=scriptsize]
	{H\left(\zeta_{n \cdot P(n)\cdot 2^{2n_0}}\right)} \\
	{F(\sqrt[m]{s})} && {F(\sqrt[m]{s}, \sqrt[m]{t})} \\
	& K \\
	{F=\mathbb{Q}(\zeta_n)} && {F( \sqrt[m]{t})}
	\arrow[no head, from=1-1, to=2-1]
	\arrow[no head, from=2-1, to=2-3]
	\arrow[no head, from=2-1, to=4-1]
	\arrow[no head, from=2-3, to=4-3]
	\arrow[no head, from=3-2, to=2-3]
	\arrow[no head, from=4-1, to=3-2]
	\arrow[no head, from=4-1, to=4-3]
\end{tikzcd}
\captionsetup{skip=2pt}
\caption{Field Diagram for Proposition \ref{prop 3.4}}
\end{figure}

\begin{proof}
By lemma \ref{lem 3.3}, The prime ideals of $K$ ramified in $F(\sqrt[m]{s},\sqrt[m]{t})$ that only be prime ideals above $2,s,t$, and $p_{i}, 1 \le i \le l$.

We first check $s$ and $t$. By lemma \ref{lem 3.3} I, the prime ideals above $s$ in $F$ is unramified in $F(\sqrt[m]{t})$. Hence, prime ideals above $s$ in $K$ are unramified in $F(\sqrt[m]{s}, \sqrt[m]{t})$. Similarly, It is easy to show the prime ideal above $t$ of $K$ is unramified in $F(\sqrt[m]{s}, \sqrt[m]{t})$. 

Without loss of generality, assume \(n\) has an odd prime factor; otherwise, we just need to examine the ramification case of 2.
We set $n^{\prime}=2^{n_0}p_2^{n_2} \cdots \cdots p_l^{n_l}$ (resp. $m^{\prime}=2^{m_0}p_{2}^{m_{1}}\cdots p_{l}^{m_{l}}$), then $n=p_1^{n_1} \cdot n^{\prime}$ (resp. $m=p_1^{m_1} \cdot m^{\prime}$) and  $\left(n^{\prime}, p_1\right)=1$ (respect $\left(m^{\prime}, p_1\right)=1$). By the choice of $t$, we know $t$ is split completely in $\mathbb{Q}(\zeta_{p_{1}^{n_{1}+1}})$, which implies 
$t \equiv 1 \bmod p_1^{n_1+1}. $
By lemma \ref{lem 3.2}, $t \in \mathbb{Q}_{p_1}^{p_1^{n_1}}$. In other words, $t$ is $p_1^{n_1}$-power in $\mathbb{Q}_{p_1}\left(\zeta_{p_1^{n_1}}\right)$, which imply  
$$ 
t\in \mathbb{Q}_{p_1}\left(\zeta_{p_1^{n_1}}\right)^{p_1^{m_1}} \text{ and   } \quad \mathbb{Q}_{p_{1}}\left(\zeta_{p_{1}^{n_{1}}},\sqrt[p_{1}^{m_{1}}]{t} \right) = \mathbb{Q}_{p_1}\left(\zeta_{p_1 ^{n_1}}\right).$$

Hence, $(1-\zeta_{p_{1}^{n_{1}}})$ is split completely in $\mathbb{Q}(\zeta_{p_{1}^{n_{1}}},\sqrt[p_{1}^{m_{1}}]{t})$. From the assumption, We have the following field diagram:

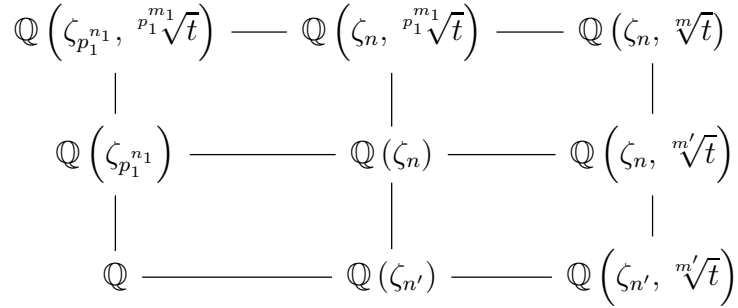
\[\begin{tikzcd}[sep=scriptsize]
	{\mathbb{Q}\left(\zeta_{p_1^{n_1}}, \sqrt[p_1^{m_1}]{t}\right)} & {\mathbb{Q}\left(\zeta_{n}, \sqrt[p_1^{m_1}]{t}\right)} & {\mathbb{Q}\left(\zeta_{n}, \sqrt[m]{t}\right)} \\
	{\mathbb{Q}\left(\zeta_{p_1^{n_1}}\right)} & {\mathbb{Q}\left(\zeta_{n}\right)} & {\mathbb{Q}\left(\zeta_{n}, \sqrt[m']{t}\right)} \\
	{\mathbb{Q}} & {\mathbb{Q}\left(\zeta_{n'}\right)} & {\mathbb{Q}\left(\zeta_{n'}, \sqrt[m']{t}\right)}
	\arrow[no head, from=1-1, to=1-2]
	\arrow[no head, from=1-1, to=2-1]
	\arrow[no head, from=1-2, to=1-3]
	\arrow[no head, from=1-2, to=2-2]
	\arrow[no head, from=1-3, to=2-3]
	\arrow[no head, from=2-1, to=2-2]
	\arrow[no head, from=2-1, to=3-1]
	\arrow[no head, from=2-2, to=2-3]
	\arrow[no head, from=2-2, to=3-2]
	\arrow[no head, from=2-3, to=3-3]
	\arrow[no head, from=3-1, to=3-2]
	\arrow[no head, from=3-2, to=3-3]
\end{tikzcd}\]

Notice the prime ideals above $p_{1}$ of $\mathbb{Q}(\zeta_{n})$ which are unramified in $\mathbb{Q}(\zeta_{n}, \sqrt[p_1^{m_1}]{t})$. Furthermore, as $p_1$ is coprime to $n^{\prime},m^{\prime}$ and \(t\), the prime ideals over $p_{1}$ in $\mathbb{Q}(\zeta_{n^{\prime}})$ are unramified in $\mathbb{Q}(\zeta_{n^{\prime}}, \sqrt[m^{\prime}]{t})$. From the field diagram, it follows that the prime ideals over $p_1$ in $\mathbb{Q}(\zeta_{n})$ are unramified in $\mathbb{Q}(\zeta_{n}, \sqrt[m]{t})$. In the same way, we deduce that the prime ideals over $p_i$ in $F$, where $p_i$ run through all the odd prime factors of $n$, are unramified in $F(\sqrt[m]{t})$. So, the prime ideals above $p_{i}$ in $K$  are unramified in $F(\sqrt[m]{s},\sqrt[m]{t})$.

On the other hands, let $n^{\prime \prime}=n /2^{n_0}$ and $m^{\prime \prime}=m/2^{m_0}$. By the choice of $t$, $ t$ is split completely in $\mathbb{Q}(\zeta_{2^{2n_0+1}})$ which mean that $t \equiv 1 \bmod 2^{2n_0+1}$. By the lemma \ref{lem 3.2}, $t\in \mathbb{Q}_{2}(\zeta_{2^{n_0}})^{2^{n_0}}$. Applying the same reasoning as previously, the prime ideals of $K$ lying over 2 remain unramified in $F(\sqrt[m]{s},\sqrt[m]{t})$.
\end{proof}

 In fact, we do not always require $t$ to satisfy such a strong splitting condition. $H$ can be substituted with any subfield of the Hilbert class field of $F(\sqrt[m]{s})$. Moreover, if $m$ is odd, the splitting condition for $t$ can be relaxed to complete splitting in $H\left(\zeta_{n \cdot P(n)}\right)$. Specifically, when $n=m=4$, we can find additional unramified construction. 

\begin{pro}
    Let $p$ and $q$ be coprime rational prime numbers, which are odd. $H$ is the Hilbert class field of $\mathbb{Q}(\zeta_{4},\sqrt[4]{p})$. If $q$ is split completely in $H(\zeta_{8})$, then 
$$ \mathbb{Q}(\zeta_{4},\sqrt[4]{p},\sqrt[4]{q}) / \mathbb{Q}(\sqrt[4]{pq},\zeta_{4})$$
is unramified everywhere. 
\end{pro}

\begin{proof}

Similarly to the proof of the previous proposition, we need to check the ramification index of 2 in $\mathbb{Q}(\sqrt[4]{q},\zeta_{4})$.
    By the choice of $q$, $q$ is split completely in $\mathbb{Q}(\zeta_{8})$, which imply $q \equiv 1 \bmod 8$. 
If $q\in \mathbb{Q}_{2}^{4}$, everything is done. 
If $q\notin \mathbb{Q}_{2}^{4}$, then 
$$ 2\mathcal{O}_{\mathbb{Q}(\sqrt[4]{q})}=P_{0}^{2^{1-t}} \cdot (P_{1}P_{2}^{2}\cdots P_{t}^{2^{t-1}})^{2^{2-t}} \text{ where } t=0 \text{ or } 1$$ by \cite{velez1988factorization}, Theorem 7. But $t=0$ is impossible. when $t=0$, we have 
$ 2\mathcal{O}_{\mathbb{Q}\sqrt[4]{q}}=P_0^2,$ which implies $\mathbb{Q}(\sqrt{q}) /\mathbb{Q}$ is not split completely at $2$. This is a contradiction with the choice of $q$. When $t=1$, $$ 2\mathcal{O}_{\mathbb{Q}(\sqrt[4]{q})}=P_{0}P_{1}^{2}.$$

Notice the ramification index of 2 in $\mathbb{Q}(\sqrt[4]{q},\zeta_{4})$ is even and $e(P_{0}|2)=1$, then $P_{0}$ is totally ramified in $\mathbb{Q}(\sqrt[4]{q},\zeta_{4})$. So the ramification index of 2 in $\mathbb{Q}(\sqrt[4]{q},\zeta_{4})$ is 2. Therefore, the prime ideals above 2 of $\mathbb{Q}(\zeta_{4})$ are unramified in $\mathbb{Q}(\sqrt[4]{q},\zeta_{4})$.  
\end{proof}

\subsection{The Proof of Theorem1.1}
Let $H$ denote the Hilbert class field of $F(\zeta_n)$. We choose \(q\) satisfying the split condition in Proposition \ref{prop 3.4} 

By the Chebotarev Density Theorem, there exist infinitely many prime $q$ such that $F(\sqrt[m]{p},\sqrt[m]{q})/F(\sqrt[m]{pq})$ is an unramified extension. Furthermore, we apply Proposition \ref{prop 2.4} to the extension \(L\) over \(H\), where $L = H(\sqrt[m]{q})$. Since the prime ideals above $q$ is totally ramified in the extension $F(\sqrt[m]{q})/F$ and unramified in the extension $H/F$.
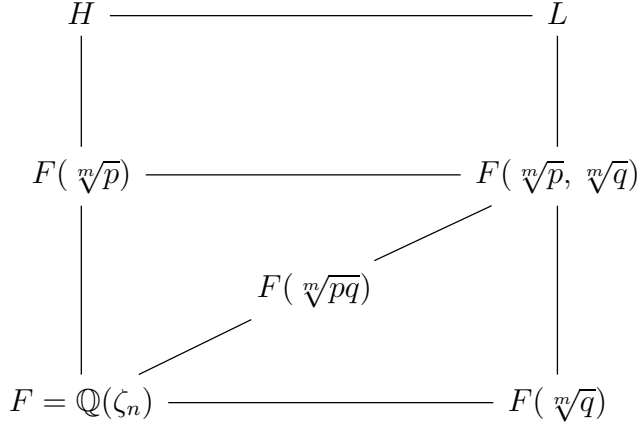
\begin{figure}[htbp]
\begin{tikzcd}
	H && L \\
	\\
	{F(\sqrt[m]{p})} && {F(\sqrt[m]{p},\sqrt[m]{q})} \\
	& {F(\sqrt[m]{pq})} \\
	{F=\mathbb{Q}(\zeta_n)} && {F(\sqrt[m]{q})}
	\arrow[no head, from=1-1, to=1-3]
	\arrow[no head, from=1-1, to=3-1]
	\arrow[no head, from=3-1, to=3-3]
	\arrow[no head, from=3-1, to=5-1]
	\arrow[no head, from=3-3, to=1-3]
	\arrow[no head, from=4-2, to=3-3]
	\arrow[no head, from=5-1, to=4-2]
	\arrow[no head, from=5-1, to=5-3]
	\arrow[no head, from=5-3, to=3-3]
\end{tikzcd}
\captionsetup{skip=2pt}
\caption{Field Diagram for Theorem \ref{thm 1.1}}
\end{figure}
Then 
\[
    \sum_{\substack{\mathfrak{p} \in R}} d_{p_i}I_{\mathfrak{P}}(K / k)\geq mh\varphi(n) d_{p_i}(\mathbb Z/m\mathbb Z)=mh\varphi(n) ,\text{where  } p_i\mid m.
\]
    By the Dirichlet's unit theorem, 
$$L_{p_i}\leq \frac{\varphi(n)mh}{2}+2\sqrt{\frac{\varphi(n)m^2h}{2}+1}.$$
As $r_1^{ram}(k)=0$, it suffices to show $$mh\varphi(n)\geq 3+\frac{\varphi(n)mh}{2}+2\sqrt{\frac{\varphi(n)m^2h}{2}+1}.$$ This implies $$\frac{\varphi(n)mh}{2}\geq \frac{3}{m}+2\sqrt{\frac{\varphi(n)h}{2}+\frac{1}{m^2}}.$$ 

In conclusion, as $h \geq \frac{12}{\varphi(n)m} + \frac{8}{\varphi(n)}$, it is established that $L$ possesses an infinite $p_i$-class field tower. Given that $K/F(\sqrt[m]{pq})$ is an unramified extension, it follows that $\mathbb{Q}(\zeta_n,\sqrt[m]{pq})$ also exhibits an infinite class field tower.

\qed

\subsection{The application of Theorem 1.1}

Apply Theorem \ref{thm 1.1} for m=3,5 and 7 to find an infinite class field tower. We can obtain the following results according to \cite{sagemath}.

\begin{cor}
    The number field $\mathbb Q(\zeta_{9},\sqrt[3]{7\cdot 181})$ has infinite 3-class field tower with root discriminant $\approx 776.7$.

\begin{proof}
    Notice the class number of $K=\mathbb{Q}(\zeta_9, \sqrt[3]{7})$ is $3$, which is greater $\frac{12}{\varphi(9)\cdot3}+\frac{8}{\varphi(9)}=2$ and 181 prime ideals split completely into the principal ideal of $K$. By Lemma \ref{lem 3.2} (I), the prime ideals above 3 in \(\mathbb{Q}(\zeta_3)\) are unramified in the field \(\mathbb{Q}(\zeta_3, \sqrt[3]{181})\). Consequently, the prime ideals above 3 in \(\mathbb{Q}(\zeta_9)\) are unramified in \(\mathbb{Q}(\zeta_9, \sqrt[3]{181})\). This leads to the same unramified result asserted in Proposition \ref{prop 3.4}. Similarly, following the proof of the main theorem, the number field \(\mathbb{Q}(\zeta_9, \sqrt[3]{7 \cdot 181})\) possesses an infinite 3-class field tower.
\end{proof}

\begin{remark}
    In \cite{leshin2014infinite}, Leshin asserts that as the class number of $K=\mathbb{Q}(\zeta_3,\sqrt[3]{79})$ equals 12 and the prime 97 splits completely into principal ideals in $K$, the field $\mathbb{Q}(\zeta_3,\sqrt[3]{79\cdot 97})$ possesses an infinite 3-class field tower. 
    
    The statement may be incorrect. Suppose $L/K$ is an abelian unramified extension of degree $n$, with $p \mid n$ and $n$ not being a power of $p$. While $L$ possesses an infinite $p$-class field tower, it is not guaranteed that $K$ does. We find an counterexample in \cite{schoof1986infinite}, let $H$ be the hilbert class field of $\mathbb Q(\sqrt{-191}) $, $L=H(\sqrt{773})$ and $K=\mathbb Q(\sqrt{-191\cdot 773})$. $L/K$ is an abelian unramified extension with degree 26. According to Schoof \cite{schoof1986infinite}, the number field L possesses an infinite 2-class field tower. But, the class group of K is isomorphic to $\mathbb{Z}/56\mathbb{Z}$, implying that K has a finite 2-class field tower, as established by Furtwangler or see \cite{lemmermeyer2010class}, Proposition 1.2.10. 
\end{remark}
\end{cor}
\begin{cor}\label{cor 3.9}
The number field $\mathbb Q(\zeta_{40},\sqrt[5]{3\cdot 41})$ has infinite $5$-class field tower with root discriminant $\approx 1196.2$.
\begin{proof}
    Since $\frac{12}{\varphi(40)\cdot 5 }+\frac{8}{\varphi(40)}\leq 1$, there is no need to verify the class number of $K=\mathbb Q(\zeta_{40},\sqrt[5]{3})$. Observe that $41$ is completely split in $K$, and the prime ideal $(1-\zeta_5)$ in $\mathbb Q(\zeta_5)$ is totally ramified in both $\mathbb Q(\zeta_{5},\sqrt[5]{41})$ and $\mathbb Q(\zeta_{5},\sqrt[5]{3})$. 
    
    We assert that $\mathbb Q(\zeta_{5},\sqrt[5]{3},\sqrt[5]{41})/\mathbb Q(\zeta_{5},\sqrt[5]{3\cdot 41})$ is an unramified extension. Consider the number field $\mathbb Q(\zeta_{5},\sqrt[5]{3\cdot 41^2})$. As the argument in Lemma \ref{lem 3.2}, since $3\cdot 41^2 \equiv 13^5 \pmod {5^3},$ the prime ideal $(1-\zeta_5)$ in $\mathbb Q(\zeta_5)$ is unramified in $\mathbb Q(\zeta_{5},\sqrt[5]{3\cdot 41^2})$, and $\mathbb{Q}(\zeta_{5},\sqrt[5]{3},\sqrt[5]{41})$ is the composition of $\mathbb{Q}(\zeta_{5},\sqrt[5]{3 \cdot 41})$ and $\mathbb{Q}(\zeta_{5},\sqrt[5]{3 \cdot 41^2})$. Thus, the extension $\mathbb Q(\zeta_{40},\sqrt[5]{3},\sqrt[5]{41})/\mathbb Q(\zeta_{40},\sqrt[5]{3\cdot 41})$ is unramified everywhere. Since $\mathbb Q(\zeta_{40},\sqrt[5]{3},\sqrt[5]{41})$ has infinite $5$-class field tower, $\mathbb Q(\zeta_{40},\sqrt[5]{3\cdot 41})$ also possesses an infinite $5$-class field tower with a small root discriminant of approximately $1196.2$.
\end{proof}
\end{cor}
\begin{cor}
The number field $\mathbb Q(\zeta_{35},\sqrt[7]{5\cdot 71})$ has infinite $7$-class field tower with root discriminant $\approx 1608.8$.
\begin{proof}
    The same argument as in Corollary \ref{cor 3.9}.
\end{proof}

\end{cor}

\bibliographystyle{plain}
\bibliography{refs}

\end{document}